\documentclass[a4paper,10pt]{amsart}
\usepackage{amsmath}\usepackage{amsthm}
\usepackage{tabu}
\usepackage{tikz-cd}
\usetikzlibrary{arrows}
\tikzset{commutative diagrams/.cd,arrow style=tikz,diagrams={>=latex}}
\usepackage{latexsym}
\usepackage[psamsfonts]{amssymb}
\usepackage[colorlinks=true,linkcolor=black,citecolor=black]{hyperref}
\newtheorem*{cor*}{Corollary}
\newtheorem{theorem}{Theorem}
\newtheorem*{theorem*}{Theorem}
\newtheorem{lemma}{Lemma}
\newtheorem{prop}[theorem]{Proposition}
\newtheorem{cor}{Corollary}
\newtheorem{thm}{Theorem}

\theoremstyle{definition}

\newtheorem*{remark*}{Remark}
\newtheorem*{example*}{Example}

\newcommand{\comment}[1]{}

\def\im{\operatorname{Im}} \def\ker{\operatorname{Ker}} 
  
 \def\Q{\mathbb{Q}}\def\Z{\mathbb{Z}}\def\RR{\mathbb{RR}}
\def\le{\leqslant} \def\ge{\geqslant}
\def\tr{\operatorname{tr}}
\def\SL{\mathrm{SL}} \def\PSL{\mathrm{PSL}}\def\GL{\mathrm{GL}}
\def\I{\mathcal{I}} \def\J{\mathcal{J}}
 \def\CC{\mathfrak{C}}
 \def\R{\mathcal{R}} \def\RR{\mathbb{R}}
\def\M{\mathcal{M}}\def\F{\mathcal F}
\def\A{\mathcal A} \def\B{\mathcal B} \def\T{\mathcal T}
\def\CC{\mathbb{C}}
 \def\a{\mathfrak a}
 \def\DD{\Delta} \def\G{\Gamma}
\def\ID{\I_\Delta} \def\JD{\J_\Delta} 
\def\RD{\mathcal{R}_\Delta}
\def\AD{\mathcal{A}_\Delta}\def\BD{\mathcal{B}_\Delta}
\def\a{\alpha}\def\b{\beta}\def\g{\gamma}\def\wa{\widetilde{a}}
\def\+{\,+\,} \def\m{\,-\,} \def\={\;=\;} \def\inn{\;\in\;}
\def\sm#1#2#3#4{\left(\begin{smallmatrix}#1&#2 \\ #3 & #4 \end{smallmatrix}\right)}
\def\ma#1#2#3#4{\left(\begin{matrix}#1&#2 \\ #3 & #4 \end{matrix}\right)}
\def\smm#1#2{\left\la\begin{smallmatrix}#1 \\ #2  \end{smallmatrix}\right\ra}
\def\smp#1#2{\left[\begin{smallmatrix}#1 \\ #2  \end{smallmatrix}\right]}
\def\be{\begin{equation}}  \def\ee{\end{equation}}
\def\wf{\widetilde{f}}
\def\H{\mathcal{H}}

\def\im{\operatorname{Im}} \def\re{\operatorname{Re}}
\def\aa{\alpha}\def\bb{\beta}

\def\dd{\delta}
\def\la{\langle} \def\ra{\rangle} \def\wT{\widetilde T} 
\def\wR{\widehat \R} 
\def\bsh{\backslash}
\def\rar{\rightarrow}
\def\bd{\circ}

\def\w{\textsf{w}}
\subjclass[2010]{11F11, 11F25, 11F67}
\title[The Eichler-Selberg trace formula]{An elementary proof of \\ the 
Eichler-Selberg trace formula}
\author{Alexandru A. Popa and Don Zagier}

\address{Institute of Mathematics ``Simion Stoilow" of the Romanian Academy,
P.O. Box 1-764, RO-014700 Bucharest, Romania}
\email{alexandru.popa@imar.ro}
\address{Max-Planck-Institut f\"ur Mathematik, Vivatsgasse 7, 53111 Bonn, Germany}
\email{don.zagier@mpim-bonn.mpg.de}

\begin{document}
\begin{abstract}
We give a purely algebraic proof of the trace formula for Hecke operators on modular forms 
for the full modular group $\mathrm{SL}_2(\Z)$, using the action of Hecke operators 
on the space of period polynomials. This approach, which can also be applied 
for  congruence subgroups, is more elementary than the classical ones
using kernel functions, and avoids the analytic difficulties inherent in the latter
(especially in weight two). 
Our main result is  an algebraic property of a special Hecke element that involves 
neither period polynomials nor modular forms, yet immediately implies both the trace 
formula and the classical Kronecker-Hurwitz class number relation. This key property 
can be seen as providing a bridge between the conjugacy classes and the right cosets 
contained in a given double coset of the modular group.
\end{abstract}
\maketitle

\section{Introduction}

Our aim in this paper is to give a short, algebraic proof of the trace formula for Hecke operators 
on modular forms for the full modular group. We use the action of Hecke operators 
on the space of period polynomials associated to modular forms, 
bringing to completion an idea introduced by the second author 25 years 
ago~\cite{Z}. A completely different (and considerably more complicated) proof 
based also on the action of Hecke operators on period polynomials was given 
in~\cite{Z2}. The proof given here depends on purely algebraic properties of a special Hecke element, 
independent of its action on period polynomials.
The same Hecke element has been used by the first author in 
two sequels to this paper to obtain simple trace formulae on modular forms for 
congruence subgroups as well~\cite{P,P2}. Our approach is more elementary
than the classical automorphic kernel method, and applies uniformly in all weights, whereas
the classical approach requires additional technicalities in weight two.

Let $\Gamma$ be the group $\mathrm{PSL}_2(\Z)$, which is generated by 
the matrices $S=\sm0{-1}10$ and $U=\sm 1{-1}10$,  modulo the relations 
$S^2=U^3=1$, and let $T=US=\sm 1101$. For $n\ge 1$, let $\M_n$ be the set 
of $2\times2$ matrices with integer entries of determinant~$n$, modulo~$\{\pm 
1\}$,
and let $\R_n=\Q[\M_n]$. The group~$\G$ acts both on the left and on the right 
on the $\Q$-vector space 
$\R_n$. It was shown in~\cite{CZ} that an element $\wT_n\in\R_n$ acts as 
the $n$th Hecke operator on period polynomials if
\be\label{A}   \tag{A}
 (1-S)\wT_n \m T_n^\infty(1-S) \inn \, (1-T)\R_n
\ee
where $T_n^\infty=\sum_{M\in \M_n^\infty} M$, for $\M_n^\infty\subset \M_n$ any 
set of representatives for $\G\bsh\M_n$ fixing~$\infty$. It was stated in 
\cite{Z} that there exists such an element $\wT_n$ which 
further satisfies the properties
\begin{equation}\label{B}\tag{B} 
\begin{cases}   \quad\qquad\wT_n(1+S) &\in \quad(1+U+U^2) \R_n\;,  \\  
                 \quad \wT_n (1+U+U^2) &\in\quad (1+S)\R_n\;,
\end{cases}
  \end{equation}
and shown that any such element would lead to an explicit formula 
for the traces of Hecke operators on modular forms on~$\G$. The proof 
of existence, omitted in~\cite{Z}, will be given below in Lemma~\ref{L4}.

If $\mathcal{S}$ is any subset of~$\M_n$ and $\xi=\sum c_MM\in \R_n$, 
we denote by $\la\xi,\mathcal{S}\ra$ the number $\sum_{M\in \mathcal{S}}c_M$. 
There are three natural actions of~$\G$ on~$\R_n$, by left multiplication, 
right multiplication, and conjugation, and we will be particularly interested 
in the cases when~$\mathcal{S}$ is an orbit with respect to the second or 
third of these, i.e., a right coset $K=M_0 \G$ or a $\G$-conjugacy class 
$X=\{\g^{-1} M_0 \g\mid  \g\in\G \}$,  respectively, where $M_0\in \M$.

The $\G$-conjugacy classes of elements $M\in\M_n$ are of five types: 
{\it scalar} ($M=\pm\sqrt n\cdot I_2$), {\it elliptic} ($\tr(M)^2<4n$), 
{\it split hyperbolic} ($\tr(M)^2-4n$ a positive square), 
{\it non-split hyperbolic} ($\tr(M)^2-4n$ a positive non-square), 
and {\it parabolic} ($\tr(M)^2=4n$, $M\ne\,$scalar). We define the {\it weight} 
$w(M)$ in these five cases by the formulae\vspace{-3mm} 
\begin{table}[h] \centering
\extrarowsep=_3pt^3pt
\begin{tabu}{c|[1.3pt]c|c|c|c|c}
$M$ & scalar & elliptic  & split hyperbolic & non-split hyperbolic  & parabolic \\ 
\tabucline[1.3pt]-
  $w(M)$ & 1/6 & $-1/|\G_M|$ & 1 & 0 & 0\\ 
\end{tabu}
\end{table}
\vspace{-3mm}

\noindent where $\G_M$ is the centralizer of~$M$ in~$\G$,
which for elliptic~$M$ has order equal to~2 or~3 if $M$ is conjugate to a 
matrix in $\Z\, I_2+\Z\, S$ or $\Z\, I_2+\Z\, U$, respectively, 
and to $1$ in all 
other cases. Note that in the last three cases we can also write 
$w(M)=1/|\G_M|$, with the convention that $1/\infty=0$ if $|\G_M|=\infty$.
Since $w(M)$ depends only on the conjugacy class~$X$ of~$M$, we can also denote 
it by~$w(X)$. 

Our main result can then be stated as follows.
\begin{theorem}\label{T0} Let $n$ be a positive integer, and let 
$\wT_n\in\R_n$ satisfy both~\eqref{A} and~\eqref{B}.
\begin{enumerate}
  \item For any right $\G$-coset $K\subset\M_n$ we have $\la\wT_n,K\ra=-1$.
  \item For any $\G$-conjugacy class $X$ we have 
  \be \tag{C}\label{C} \la\wT_n,X\ra \= w(X) \;.\ee  
\end{enumerate}
\end{theorem}

We will show in the next section that the theorem immediately implies 
the Eichler-Selberg trace formula for modular forms for $\G$. As a warm-up, 
and to introduce the class numbers, we observe that computing $\la\wT_n,\M_n 
\ra$ in two ways, using parts (i) and (ii) of the theorem,
yields the Kronecker-Hurwitz class number relation in the form:
\be\label{0}
\sum_{X\subset \M_n} w(X)\=-|\G\bsh\M_n|\;, \ee
where the sum is over all conjugacy classes $X$. To bring the formula to 
its classical form, we use the $\G$-equivariant bijection 
$\sm abcd \leftrightarrow cx^2 +(d-a) xy -b y^2$
between integral matrices of determinant~$n$ and trace~$t$ 
and quadratic forms of discriminant $t^2-4n$ to write: 
\be\label{HD}
\sum_{\substack{X\subset \M_n\\ \tr(X)=\pm t }} w(X) \= 
\begin{cases}
-2H(4n-t^2)  & \text{ if  } t\ne 0 \;,\\
- H(4n-t^2) & \text{ if  } t= 0 \;,
\end{cases}
\ee
where $\tr(X)$ is the trace of any element in the conjugacy class $X$ 
(well-defined up
to sign), and $H(D)$ is the Kronecker-Hurwitz class number, extended to all 
$D\in \Z$ as in~\cite{Z}. 
That is, for $D>0$, $H(D)$ equals the number of $\G$-equivalence classes 
of  positive definite integral binary quadratic forms of discriminant~$-D$, 
with those classes that contain a multiple of $x^2+y^2$ or of $x^2-xy+y^2$ 
counted with multiplicity $1/2$ or $1/3$, respectively, $H(0)=-1/12$, $H(D)=0$
if $D<0$ is not the negative of a perfect square, and $H(-u^2)=-u/2$ if $u\in 
\Z_{>0}$. 

Using~\eqref{HD}, the relation above becomes $\sum_{t\in \Z} H(4n-t^2)=\sum_{d|n} d$, and 
the classical Kronecker-Hurwitz relation 
\[\sum_{t^2\le 4n} H(4n-t^2) \= \sum_{n=ad,\; a,d>0} \max(a,d)
\]
follows by observing that 
$\sum_{t^2>4n} H(4n-t^2) = \sum_{n=ad, d>a>0} a-d $. We gave a yet 
different, simpler proof of (a refinement of) the Kronecker-Hurwitz formula in~\cite{PZ}. 
As a by-product of the proof of Theorem~\ref{0}, we will obtain a different 
refinement here (Proposition~\ref{P5}).

We prove part (i) of the theorem in Section~\ref{S2} (Corollary~\ref{C1}), 
after a preliminary study of the relation between operators~$\wT_n$ satisfying 
only one of properties \eqref{A},~\eqref{B}. The hardest part of this approach 
to the trace formula is the proof of part~(ii), given in Section~\ref{S3}. We give an 
explicit element $\wT_n$ satisfying part~(ii) by construction, and then show that it 
satisfies part (i), as well as property~\eqref{B}. It then follows from the 
theory in Section~\ref{S2} that $\wT_n$ satisfies~\eqref{A} as well, and Theorem~\ref{T0}
follows.

The proof of the trace formula in Section~\ref{s2} does not require the 
full statement of Theorem~\ref{T0}, but only the existence of an element $\wT_n$ 
satisfying properties \eqref{A}, \eqref{B} and \eqref{C}.  
For example, to finish the proof given in Section~\ref{s2} for 
the trace of the first two Hecke operators, it is enough 
to check that the following elements satisfy~\eqref{A}--\eqref{C}:
 \[ \wT_1=I_2-\frac 12 (I_2+S)-\frac 13 (I_2+U+U^2)\;, \] 
\[
\wT_2= \ma 2001-\frac 12 \ma 11{-1}1 -\frac 12 \ma 1{-1}{1}1 -\ma 1{-1}20 
-\ma 02{-1}1-\ma 0{-2}10\;.
\]
(Both of these expressions were given in~\cite{Z}; the element $\wT_2$ constructed in 
Section~\ref{S3} is different from the above).  

While the methods of this paper are elementary, we point out that the formula
we obtain has a cohomological interpretation with wide-ranging generalizations. 
Let $V$ be a finite-dimensional complex $\SL_2(\RR)$-module. 
For a $\G$-double coset $\DD$ we denote by $[\DD]$ the corresponding action on 
the group cohomology $H^i(\G,V)$. Then we have the following formula for the 
Lefschetz number of the correspondence on the modular surface determined by the 
double coset~$\DD$:
\be\label{3}
\sum_i (-1)^i \tr([\DD], H^i(\G,V)) \= -\sum_{X\subset \DD/\{\pm 1\}} w(X) 
\tr(M_X, V) 
\ee
where the sum is over $\G$ conjugacy classes $X$ with representatives 
$M_X\in\DD$. The results of this paper can be interpreted as proving this 
formula for irreducible representations $V$, and therefore for all 
finite-dimensional representations. Indeed, for the trivial module $V=\CC$ only 
$H^0$ is nonzero and the formula reduces to the Kronecker-Hurwitz
relation~\eqref{0}. If $V=V_\w$ is the unique irreducible representation of 
$\SL_2(\RR)$ of odd dimension $\w+1\ge 3$, only $H^1$ is nonzero, and the 
formula above is proved in~\eqref{1}, taking into account the 
Eichler-Shimura isomorphism $H^1(\G,V_\w)\simeq S_{k}\oplus M_{k}$, 
with $M_{k}$, $S_k$ the spaces of modular forms, respectively cusp forms of 
weight $k=\w+2$ for~$\G$.
In a sequel to this paper~\cite{P}, the first author has
proved the cohomological trace formula~\eqref{3} for arbitrary congruence 
subgroups $\G$ of the modular group, under a mild assumption on the double coset 
$\DD$. Surprisingly, Theorem~\ref{T0} is central for the proof in the congruence
subgroup case as well. 

Another case where the formula~\eqref{3} is known for more general groups~$\G$ 
is when~$\DD=\G$, the trivial double coset, in which case the left hand side reduces to an 
Euler-Poincar\'e characteristic. For a large class of groups~$\G$ including
arithmetic subgroups of linear groups, it follows from work of Bass and 
Brown~\cite{Ba,Br} that (assuming for simplicity that~$\G$ has trivial center)
\[\sum_i (-1)^i \dim H^i(\G,V) =\sum_{M\in T(\G)} \chi(\G_M)\tr(M,V)
\]
where $T(\G)$ is a set of representatives for the conjugacy classes of $\G$,
$\G_M$ denotes the centralizer of~$M$, and  $\chi(G)\in \Q$ denotes the 
homological Euler-Poincar\'e characteristic of the group~$G$. The sum 
on the right hand side actually runs over finite order elements in $T(\G)$, 
because of a theorem of Gottlieb-Stalling that states that~$\chi(G)=0$ 
if the group~$G$ contains an infinite order element 
in the center. Moreover, if~$G$ is finite we have $\chi(G)=1/|G|$, and  
$\chi(\PSL_2(\Z))=-1/6$,  so for $\G=\PSL_2(\Z)$ the formula specializes 
to the case $\DD=\G$ of~\eqref{3}. Therefore we have a 
geometric interpretation of the coefficients $w(X)$ in~\eqref{3}--they are 
negatives of ``local'' Euler-Poincar\'e characteristics.  

Perhaps the ultimate generalization of~\eqref{3} is the topological trace 
formula of Goresky and MacPherson~\cite{GM}, where $\G$ can be 
any arithmetic subgroup of a reductive group, and the algebraic group cohomology 
can be replaced by other geometric cohomology theories. Our work can therefore 
also be seen as giving an elementary proof of an explicit version of the topological 
trace formula, in the special case of the modular group~$\PSL_2(\Z)$.

\section{Deduction of the trace formula}\label{s2}

Let $\w\ge 0$ be an even integer, and let $V_\w$ be the space of complex 
homogeneous polynomials of degree $\w$. The group $\GL_2(\RR)$ acts on $V_\w$ 
by $P|\g (X,Y)=P(aX+bY, cX+dY)$ for $\g=\pm\sm abcd$, which extends 
by linearity to a right action of the algebra $\R=\bigoplus_{n\ge 1} \R_n$, where 
$\R_n=\Q[\M_n]$ as above. The space of period polynomials $W_\w$ is defined as 
$W_\w=\ker(1+S)\cap \ker(1+U+U^2)\,. $

For even $k>2$, let $M_{k}$, $S_{k}$ be the spaces of modular forms, 
respectively  cusp forms of weight $k$ for $\G$, so we have 
$M_{k}=\CC G_{k} \oplus S_{k}$, with $G_k$ a suitably normalized 
Eisenstein series of weight $k$. To each $f\in S_{k}$ we associate 
its period polynomial $P_f(X,Y)=\int_0^{i\infty}f(z) (X-zY)^{\w} dz$, where 
$\w=k-2$, which is an
element of $W_{\w}$. Its even and odd parts, $P_f^+$ and $P_f^-$, also belong 
to $W_{\w}$. One can also define $P_f^+$ for non-cuspidal $f$ (see~\cite{Z1}),
and we have $P_{G_k}^+(X,Y)=\lambda_k (X^{\w}-Y^{\w})$ for a certain number 
$\lambda_k\ne 0$.

For $n\ge 1$, the Hecke operator $T_n$ acts on $M_{k}$ by 
$f|T_n=n^{k-1}\sum_{M\in\G\bsh \M_n} f|_{k}M$, with 
$f|_k M (z)=f(Mz)(cz+d)^{-k}$ for $M=\sm **cd \in \GL_2(\RR)$. The 
Eichler-Shimura isomorphism 
can then be stated as follows. For a generalization to modular forms 
for congruence subgroups see~\cite{PP}. 
\begin{prop}\label{P1} Let $\w>0$ be even. We have a Hecke-equivariant 
isomorphism 
\[ M_{\w+2}\oplus S_{\w+2}\;\simeq\; W_\w \,, \quad 
 (f,g)\mapsto P_f^+ + P_g^-\,\]  
where the action of the Hecke operator $T_n$ on period polynomials $P\in W_\w$
is defined by $P|T_n=P|\wT_n$ for any element $\wT_n\in \R_n$ 
satisfying~\eqref{A}.  
\end{prop}
\begin{proof} 
The fact that the map is an isomorphism and that $P_{f|T_n}=P_f|\wT_n$ for 
$f\in S_{\w+2}$ and any $\wT_n\in \R_n$ satisfying~\eqref{A}  is shown 
in~\cite{CZ} or~\cite{Z}. For completeness we briefly sketch the proof
of Hecke-equivariance: The Eichler integral $\wf(z)=\int_{z}^{i\infty} 
f(t)(t-z)^\w dt$  (with $z\in \H$) has the property that 
$\wf|_{-\w}(1-S)=P_f(z,1)$ and 
$\wf|_{-\w}(1-T)=0$, so for any $\wT_n$ satisfying~\eqref{A} we have:
\[ P_f(z,1)|T_n\=P_f(z,1)|\wT_n=\wf|_{-\w} (1-S)\wT_n\=\wf|_{-\w}T_n^\infty(1-S)=P_{f|T_n}(z,1),
\]
where the last equality follows from the easily verified equality
$\wf|_{-\w} T_n^\infty= \widetilde{f|T_n}$. 

For the period polynomial $X^\w-Y^\w=Y^\w|(S-1)$
of the Eisenstein series, using~\eqref{A} and the fact that $Y^\w|(1-T)=0$  we 
obtain 
\[  Y^\w|(S-1)|T_n\= Y^\w|(S-1)\wT_n\=Y^\w|T_n^\infty (S-1)\=\sigma_{\w+1}(n) 
Y^\w|(S-1)\, , \]
where $\sigma_{\w+1}(n)=\sum_{d|n}d^{\w+1}$ is the eigenvalue of $G_{\w+2}$
under $T_n$. 
\end{proof}
The next result, proved in~\cite{Z}, is the reason for introducing 
property~\eqref{B}. For completeness we give here a different, shorter proof. 
\begin{prop}\label{P2} Let $\w>0$ and $n\ge 1$ be integers with $\w$ even. If
$\wT_n\in \R_n$ satisfies~\eqref{B}, then $\wT_n$ preserves the subspace 
$W_\w$ of $V_\w$ and satisfies
$$\tr(\wT_n, W_\w)\= \tr(\wT_n, V_\w) \,.$$
\end{prop}
\begin{proof}
Property~\eqref{B} implies that the operator~$\wT_n$ maps the subspaces 
$A=\ker(1+S)$ and $B=\ker(1+U+U^2)$ of $V_\w$ into each other, so 
it maps $W_\w=A\cap B$ into itself. On the other hand $A+B=V_\w$, because 
$V_\w$ is endowed with a natural, nondegenerate $\G$-invariant inner product, 
and the orthogonal complement $(A+B)^\perp$ is the $\G$-invariant subspace 
$V_\w^\G=\ker(1-S)\cap\ker(1-U)$, which is trivial for $\w>0$. The claim now follows
immediately from a simple linear algebra fact: if $T$ is a 
linear transformation of a vector space to itself, and $A,B$ are subspaces 
mapped into each other by~$T$, then $\tr(T,A\cap B)=\tr(T,A+B)$. 
\end{proof}

To state the trace formula, let 
$p_\w(t,n)$ be the Gegenbauer polynomial, defined by the power series expansion
$(1-tX+nX^2)^{-1}=\sum_{\w=0}^\infty p_\w(t,n)\,X^\w$. It satisfies 
$\tr(M,V_\w)=p_{\w}(\tr M,\det M)$ for any $M\in\GL_2(\RR)$ 
(in particular $p_\w(t,n)$ is an even function of $t$ for $\w$ even). 
\begin{cor}[The Eichler-Selberg trace formula]\label{C0}
For all $\w>0$ and $n\ge 1$ we have  
   \[
   \tr(T_n, M_{\w+2})\+ \tr (T_n, S_{\w+2})\=
   -\sum_{t \in \Z} p_\w(t,n)\; H(4n-t^2)\;.
   \]
\end{cor}
\begin{proof} For odd $\w$ both sides vanish trivially, so we assume $\w$ even. 
Let $\wT_n\in \R_n$ satisfy \eqref{A}~and~\eqref{B}. Combining the two 
propositions and part (ii) of Theorem~\ref{T0}, we obtain for even $\w\ge 2$: 
\be\label{1}
\begin{aligned}
\tr(T_n, M_{\w+2})\+ \tr (T_n, S_{\w+2})&\=\tr(\wT_n, W_\w ) 
\= \tr(\wT_n, V_\w)\\  &\=  \sum_{X\subset \M_n} \tr(M_X, V_\w)\; w(X)\;,
\end{aligned}
\ee
where the last sum is over all conjugacy classes $X$, and $M_X$ 
is any element in~$X$. 
The conclusion follows by rewriting the last sum using the property of 
$p_\w(t,n)$ above, 
together with formula~\eqref{HD}.
\end{proof}  
 Note that $\tr(T_n, M_{k})-\tr(T_n, S_k)=\sigma_{k-1}(n)$, so that 
Corollary~\ref{C0} can be rewritten as a formula for either  
$\tr(T_n, M_{k})$ or $\tr(T_n, S_k)$ ($k>2$ even), which is the 
form usually given in the literature. 

\begin{remark*}The formula in Corollary~\ref{C0} is generalized to  modular 
forms on congruence subgroup~$\G$ with Nebentypus in~\cite{P2}, 
using the same operator $\wT_n$ as in Theorem~\ref{T0} acting on the 
space of period polynomials for~$\G$. The trace on the Eisenstein subspace is also 
explicitly computed there, yielding a simple formula for the trace of a 
composition of arbitrary Hecke and Atkin-Lehner operators on cusp forms for~$\G_0(N)$. 
\end{remark*}

\section{General properties of Hecke operators}\label{S2}

In this section we analyze and generalize properties \eqref{A} and~\eqref{B} 
of the introduction for arbitrary double cosets of $\G$ in the set~$\M$ of
$2\times 2$ matrices with integer entries and positive determinant, modulo~$\{\pm 1\}$.
We let $\R=\Q[\M]$, which is a left and right module for the action of 
the group ring $\R_\G=\R_1=\Q[\G]$.

Generalizing the definition of the subspace $W_\w\subset V_\w$ of
period polynomials, we can consider for any (right)  $\R_\G$-module 
its {\it period subspace} $\ker \pi_S \cap \ker\pi_U$, where 
\[\pi_S\=\frac{1+S}2 \,,\quad \pi_U\=\frac{1+U+U^2}3  \]
are the idempotents corresponding to the generators $S$ and $U$ of $\G$ of
order~2 and~3. Equivalently, the period subspace of a (right) $\R$-module is 
the space annihilated by the right $\R_\G$-ideal
  \be \label{I}  \I \= \pi_S \R+\pi_U\R \,, \ee
and the most general subset of $\R$ which preserves 
the period subspace is 
  \[ \A \= \bigl\{ \xi\in \R \mid \xi\I\subset \I  \bigr\} 
\=\bigl\{ \xi\in \R  \mid \xi\pi_S\in \I\,,\;  \xi\pi_U\in \I  \bigr\}\,. \]
Note that elements $\wT_n$ satisfying property~\eqref{A} (and of course their 
multiples) belong to~$\A$. Indeed multiplying~\eqref{A} by $\pi_S$, $\pi_U$, 
and using the relation $(1-S)\pi_U\=(1-T^{-1})\pi_U$ and the fact that 
$T_n^\infty(1-T)\in (1-T)\R$, we obtain
$$(1-S)\wT_n\pi_S\,,\;\, (1-S)\wT_n\pi_U\;\,\in\; (1-T)\R\, ,$$ 
so $\wT_n\pi_S,\, \wT_n\pi_U\in\I$ by the characterization of $\I$ in the next 
lemma. 
 \begin{lemma}[\cite{CZ}] \label{L2} We have 
 $\; \xi\inn \I \; \Longleftrightarrow \;  (1-S)\xi \inn (1-T)\R\, .$
\end{lemma}
 
For the purpose of proving the trace formula, we introduce 
the subset~$\B$ of $\A$ given by
  \be \label{Bb}\B\=\{ \xi \in \R \mid  \xi\pi_S\in \pi_U\R,
  \quad  \xi\pi_U\in \pi_S\R  \}\,, \ee
which contains those elements satisfying relation~\eqref{B}. 
To show that the set $\B$ contains elements satisfying~\eqref{A} as well,
we need a preliminary lemma. 
\begin{lemma}\label{L1} The action of $\pi_S$ and $\pi_U$ on $\R$ satisfies
$$\im(\pi_S)\cap\im(\pi_U)\=\ker(\pi_S)\cap\ker(\pi_U)\=\{0\}\;.$$
\end{lemma}
\begin{proof}
The first statement is clear since any element in $\pi_S\R\cap\pi_U\R $ is left
invariant under both $S$ and $U$, hence is $\G$-invariant, and hence equal to 0. 
The second property, called ``acyclicity'' in~\cite{CZ}, was proved 
there using the action of $\R$ on rational period functions. For completeness
we give a shorter, more direct proof here.

Assume that $\xi\in \ker(\pi_S)\cap\ker(\pi_U)$, so in particular
$(1+S)\xi=(1+U+U^2)\xi$. Setting $T'=U^2S=\sm 1011$, and recalling that  $T=US$ 
we obtain $\xi=(T^{-1}+T'^{-1})\xi$. Letting $\xi=\sum c(M)M$, it follows 
that $c(M)= c(TM)+ c(T' M)$ for all $M\in \M$, and this immediately 
leads to a contradiction if $c(M)\ne0$: we would get $c(\g_iM)\ne0$ 
for an infinite sequence of elements $\g_i\in\G$ with $\g_0=1$ and 
$\g_{i+1}=T\g_i$ or $T'\g_i$, and this is impossible since the 
$\g_i$ are distinct (they have non-negative coefficients whose 
sum increases strictly), and hence the $\g_iM$ are also all distinct. Therefore 
$\xi=0$. 
\end{proof}

\begin{lemma}\label{L4} 
For $\xi\in\A$, there are unique elements $\xi_S\in \pi_S \R$, $\xi_U\in \pi_U 
\R$ such that 
\[ \xi \pi_S\m \xi_S\inn \pi_U \R\,,
\quad \xi \pi_U\m \xi_U\inn \pi_S \R\,.\]
The map $P:\A\rightarrow \A$ given by $\,P(\xi)\, :=\, \xi \m \xi_S \m 
\xi_U\, $
is a projection onto $\B$, and the image of the ideal $\I$ under 
this map is the set $\J\subset \I$ given by
\begin{equation}\label{9}
 \J \= \pi_S\R(1-\pi_S)\+\pi_U\R(1-\pi_U)\,.
\end{equation}
\end{lemma}
Note that the projection $P:\A\rightarrow \B$ defined in the lemma satisfies
 $P(\xi)-\xi\in\I$.  This proves the existence of elements satisfying both~\eqref{A} 
 and~\eqref{B}, since any element~$\wT_n\in\R_n$ satisfying~\eqref{A} belongs to~$\A$ 
 as we saw above, and  $P(\wT_n)\in \wT_n+\I$ satisfies both~\eqref{A} and~\eqref{B}.
\begin{example*} 
We have $1\in \A$, and $1_S=\pi_S$, $1_U=\pi_U$, so $P(1)=1-\pi_S-\pi_U\in\B $.
\end{example*} 
\begin{proof} 
The existence of $\xi_S$, $\xi_U$ follows from the definition of $\A$, 
and their uniqueness from the first part of Lemma~\ref{L1}. The element
$P(\xi)$ belongs to $\B$ because $P(\xi)\pi_S=(\xi\pi_S-\xi_S-\xi_U)\pi_S\in 
\pi_U\R$ (using $\pi_S=\pi_S^2$), and similarly for  $P(\xi)\pi_U$. 
If $\xi\in \B$, the uniqueness of $\xi_S$, $\xi_U$ shows that $\xi_S=\xi_U=0$, 
so $P(\xi)=\xi$.

We have $P(\I)=\I\cap \B$, which clearly contains $\J$. To show the 
other implication, let $\xi\in\I\cap \B$, that is, $\xi=\pi_S X+\pi_U Y $, 
with $\pi_S X \pi_S\in \pi_U \R$ and $\pi_U Y \pi_U\in \pi_S \R$. 
By the first part of Lemma~\ref{L1} we get $\pi_S X \pi_S=0$, and since 
$\ker(\pi_S)=\im(1-\pi_S)$, we have  $\pi_S X\in\R (1-\pi_S)$. Multiplying on 
the left by $\pi_S=\pi_S^2$, we obtain $\pi_S X\in \pi_S \R(1-\pi_S)$. 
Similarly one shows 
$\pi_U Y\in \pi_U \R(1-\pi_U)$, so $\xi\in \J$.
\end{proof}

We now come to the main results of this section, for which we need to consider
elements $\wT_\DD$ satisfying analogues of~\eqref{A} and~\eqref{B}, defined for 
any $\DD$ which is a double coset $\G\dd\G$ ($\dd\in \M$), or more 
generally a left- and right-$\G$-invariant subset of $\M_n$ (which is 
a finite union of double cosets).
 
We set $\RD=\Q[\DD]$, $\AD=\A\cap \RD$, $\BD=\B\cap \RD$.
As in the case $\DD=\M_n$, we define 
$T_\DD^\infty=\sum_{M\in \M_\DD^\infty} M \in \R_\DD,$
where $\M_\DD^\infty$ is a set of representatives for $\G\bsh \DD$ 
fixing~$\infty$. The set $\DD$ gives rise to a Hecke 
operator $T_\DD$ acting on modular forms $f\in M_k$ by
$f|T_\DD=n^{k-1} \sum_{M\in \G\bsh\DD} f|_k M$
 (for $\DD\subset \M_n$). 
One can show as in the case $\DD=\M_n$ that the corresponding action on
period polynomials is by any element $\wT_\DD$ satisfying 
  \be\tag{A$'$}\label{Aa} (1-S)\wT_\DD \;\equiv\;  T_\DD^\infty(1-S) 
\pmod{ (1-T)\R_\DD}\; , \ee
and the existence of such elements follows as in~\cite{CZ}.
The same proof given above for $\wT_n$ shows that any such $\wT_\DD$ belongs to 
$\A_\DD$. The next theorem shows that for a single double coset the 
converse is also true up to scaling. 

\begin{thm}\label{T1} Let $\DD$ be a double coset. Then any 
element $\T\in\AD$ satisfies
\be \label{90}
(1-S)\T \;\equiv\; \a  \cdot T_\DD^\infty(1-S) 
\pmod{ (1-T)\R_\DD}
\ee
for a unique number $\a=\a(\T)\in\Q$.  
\end{thm}
Equation~\eqref{90} says that $\T\equiv \a \wT_\DD \pmod{\I}$ for any 
$\wT_\DD$ satisfying~\eqref{Aa}. The theorem can be reformulated as the 
existence of a linear map $\a:\AD\rightarrow \Q$ and of an exact sequence
\[0\xrightarrow{\phantom{a\a a}}\ID \xrightarrow{\phantom{a\a a}}\AD 
\xrightarrow{\phantom{a}\a\phantom{a}} \Q \xrightarrow{\phantom{a\a a}} 0
\,,\]
where $\a$ is uniquely determined by the normalization  $\a(\wT_\DD)=1$ for 
any element~$\wT_\DD$ satisfying~\eqref{Aa}.
\begin{proof} Suppose $\T\in \A$ (we will assume $\T\in\AD$ 
only at the end of the proof). By Lemma~\ref{L2}, we get
  \be\label{4} (1-S)\T\pi_S\;\in\; (1-T)\R\,,\quad (1-S)\T\pi_U\;\in\; 
(1-T)\R\,.\ee

 Let $\G_\infty$, $\M_\infty$, and $\DD_\infty$ be the subsets fixing $\infty$ 
of $\G$, $\M$, and $\DD$, respectively, and for $\G_\infty$-orbits  
$K\in \G_\infty\bsh\M$, we denote by $M_K\in K$ a fixed representative. 

For any $\zeta\in \R$ we have  
$\zeta-\sum_{K\in \G_\infty\bsh\M}\la \zeta, K\ra  M_K\in (1-T)\R$, so      
  \be \label{T} \zeta \in (1-T)\R\  \Longleftrightarrow\ \la \zeta,K\ra =0\,,  
  \text{ for all } K\in \G_\infty\bsh\M  \,, \ee 
where the notation $\la\cdot, \cdot \ra$ was defined in the introduction. 

Set $\xi=(1-S)\T$, and define the function $a=a_\xi:\G_\infty\bsh 
\M\rightarrow \Q$ by  
$a(K)=\la \xi,K\ra$, so $a(K)$ is nonzero for finitely many $K$. 
From~\eqref{4} and~\eqref{T} we obtain
  $$a(K)\+a(KS)=0\,,\quad a(K)\+a(KU)\+a(KU^2)\=0\,.$$
It follows that $a(K)=a(KT)+a(KT')$. Since elements in $K$ share the same 
second row, we conclude as in the proof of Lemma~\ref{L1} that 
 $a(K)=0$ unless one of the elements on the second row 
of matrices in $K$ equals  0,
that is, unless $K\subset\M_\infty$ or $K\subset\M_\infty S$.
For $K\subset\M_\infty$ it follows that $a(KT')=0$, so $a(K)=a(KT)$ 
for $K\in\G_\infty\bsh\M_\infty$. 

The set of orbits  $\G\bsh\M$ can be identified with  
$\G_\infty\bsh\M_\infty$ by chosing for each $K\in \G \bsh \M$ 
a representative $M_K\in \M_\infty$.  With this identification, 
the function $a$ gives rise to a function $\wa$ on $ \G\bsh\M$, by 
$\wa(K)=a(\G_\infty M_K)$, and we have shown that $\wa(K)=\wa(KT)$ and 
  \be \label{6} \xi \equiv \sum_{K\in\G\bsh\M} 
\wa(K)M_K(1-S) \pmod{(1-T)\R}\,.\ee 
 Recalling that 
$\xi=(1-S)\T$, we have $\la \xi,K\ra=0$ for all $K\in \G\bsh\M$. Therefore
$\wa(K)=\wa(KS)$. 

We now assume that $\T\in\AD$, so the sum in \eqref{6} is over 
$K\in\G\bsh\DD$. Since~$\DD$ 
is a double coset, the group~$\G$  acts transitively on $\G\bsh\DD$ on the 
right. Since $S$ and $T$ generate $\G$, from $\wa(K)=\wa(KS)=\wa(KT)$
it follows that $\wa(K)=\a$ is constant for $K\in \G\bsh\DD$.
We conclude from~\eqref{6} that $\T$ satisfies~\eqref{90}, and then
Lemma~\ref{L2} shows that $\T\equiv \a\wT_\DD\pmod{\I}$ for any $\wT_\DD$ 
satisfying~\eqref{Aa}.

For the uniqueness of $\a$, note from Lemma~\ref{L2} that any two 
elements $\wT_\DD$ satisfying~\eqref{Aa} differ by an element in $\I_\DD$, so it 
is enough to show that $\wT_\DD\not\in\I_\DD$. But   
$\wT_\DD\in\I_\DD$ would imply that   $T_\DD^\infty (1-S)\in (1-T)\R_\DD$, 
which is easily seen to contradict~\eqref{T}.
\end{proof}
\begin{remark*}
An equivalent formulation of condition \eqref{Aa} is due to L. Merel (for the
case $\DD=\M_n$), and it was used as the definition of Hecke operators acting 
on modular symbols in~\cite{Me}. For a matrix $M=\sm abcd$, we denote its 
adjoint by $M^\vee=\sm d{-b}{-c}a$, and we extend the notation to elements 
of~$\R$ by linearity. Then the operator $\T\in \RD$ satisfies $\eqref{Aa}$ if 
and only if its adjoint $\T^\vee$ satisfies Merel's condition:
  \[ \T^\vee_K ([0]-[\infty]) \= [0]-[\infty]\,, 
   \quad \forall K \in \DD/\G\,,\]
where $\T^\vee_K=\sum_{M\in K }c_M M$ if  $\T^\vee=\sum c_M M$. Here $\M$,
and by linearity $\R$, act on the set of divisors 
supported on the cusps by fractional linear transformations. 

Indeed, after taking adjoints condition \eqref{Aa} is equivalent with 
  \[ \T^{\vee}_K(1-S)\m (M_K-S M_{SK} ) \inn M_K\cdot \R_\G 
(1-T)\quad 
  \text{for all $K\in \DD/ \Gamma$ } \;,\]
where we denote by $M_K$ the unique representative of the 
coset $K$ which is in $(\M_\DD^\infty)^\vee$. 
The previous relation is equivalent to Merel's condition by the following 
immediate consequence of the adjoint of equivalence~\eqref{T}: 
for $\xi\in \R_\G$ we have $\xi[\infty]=0$ if and only if $\xi\in  \R_\G(1-T).$ 
\end{remark*}

We now study more closely elements in the subspace $\BD=\B\cap 
\Q[\DD]$ of the vector space~$\B$. 

\begin{thm} \label{T2} Let $\DD\subset \M$ be a double coset and let 
$\T\in\BD\subset \AD$.

\emph{(a)} There exists $\b(\T)\in\Q$ such that 
$\la \T,K \ra=\b(\T)$ for all~$K\in \DD/\G$.

\emph{(b)} We have $\b(\T)=-\a(\T)$, with $\a(\T)$ 
defined in Theorem~\ref{T1}.  

\noindent Moreover, the map $\b:\BD\rar\Q$ is surjective.  
\end{thm}
The theorem  implies that the exact sequence stated after 
Theorem~\ref{T1} can be a completed to a commutative diagram with exact rows:
\[
\begin{tikzcd}
 0\arrow{r} &\ID\arrow{r} &\AD\arrow{r}{\a} &\arrow{r}\Q\arrow[equals]{d}  &0 \\
 0\arrow{r} &\JD\arrow{r}\arrow[hookrightarrow]{u}  
            &\BD\arrow{r}{-\b}\arrow[hookrightarrow]{u} &\Q\arrow{r} & 0  
\end{tikzcd}
\]
where $\J$, which was defined in~\eqref{9}, was already shown to equal 
$\I\cap\B$ in Lemma~\ref{L4}. 

\begin{proof} (a) For $\T\in\RD$, denote by $\T_K$ the part of 
$\T$ supported on a coset $K\in \DD/\Gamma$. For $\T\in\BD$, we have
$(1-U)\T\pi_S=(1-S)\T \pi_U=0$, hence
$$\T_K\pi_S\=U\T_{U^2 K}\pi_S\,,\quad \T_{K}\pi_U\=S\T_{S K}\pi_U\,,$$
for all cosets $K\in \DD/\Gamma$. It follows that 
$\la \T, K \ra =\la  \T,UK \ra=\la \T, SK \ra$ 
for all cosets~$K$. Since~$U$, $S$ generate $\Gamma$, which acts transitively 
on the cosets $\DD/\Gamma$, we obtain that $\la \T, K\ra$ is the same for 
all cosets~$K$.\vspace{2mm}

(b) To show $\bb=-\aa$, and to prove surjectivity of $\beta$, we give
a direct construction of elements $\T\in \BD$ having prescribed $\bb(\T)=\bb$.
Without loss of generality assume $\bb=1$, as we can always scale elements in 
$\BD$. Choose $A,A'\in \RD$ such that
  \be\label{10} \la A,K \ra=\la A', K\ra =1\,, \ \ \forall\, K\in 
\DD/\Gamma\,, \text{ and }
\begin{cases}   \quad \phantom{'}A \pi_S\  \in  \pi_U \RD\;, & \ \\  
                 \quad A' \pi_U\  \in \pi_S \RD\;. & \
\end{cases}
\ee
(For example, letting $T_\DD=\sum_{K\in \DD/\G} M_K$, $T_\DD'=\sum_{K\in 
\DD/\G} M_K'$ with arbitrary
representatives $M_K, M_K'\in K$, we can choose $A= \pi_U T_\DD$, $A'= \pi_S 
T_\DD'$.)
Since $S$ and $U$ generate~$\G$, there are $B,B'\in\RD$ such that 
$ A'-A=B(1-S)-B'(1-U)$, and it follows immediately that 
\be\label{11} \T=A+B(1-S)=A'+B'(1-U) \ee
belongs to $\BD$, and $\bb(\T)=\bb(A)=1$.

We now prove exactness in the middle of the second row in the diagram above. 
Let
$\xi\in \BD$, and we show that $\xi\in\JD$. Indeed,  
we have $\la \xi, K \ra =0$ for all cosets $K\in \DD/\Gamma$, and since $S$ 
and $U$ generate~$\G$ we have 
\[ \xi \=A(1-\pi_S)\+B(1-\pi_U)\,,\]
for some $A,B\in\R$. Since $\xi$ satisfies the second relation in~\eqref{Bb}, 
setting $X=(1-\pi_S)A(1-\pi_S)$ we have $X\pi_U=0$. But $X\pi_S=0$ as well, 
and we conclude from Lemma~\ref{L1} that $X=0$. Since $\ker(1-\pi_S)=\im 
\pi_S$, 
it follows that $A(1-\pi_S)\in \pi_S\R$, and by multiplying on the right by
$(1-\pi_S)=(1-\pi_S)^2$, we obtain $A(1-\pi_S)\in \pi_S\R(1-\pi_S)$. 
Similarly one shows $B(1-\pi_U)\in \pi_U\R (1-\pi_U)$, and we conclude that 
$\xi\in \J$.

Coming back to the proof of $\a=-\b$, we assume without loss 
of generality that $\DD=\G \sm 100n \G$. All operators 
$\T\in \BD$ with $\b(\T)=1$ differ by an element in $\JD$, by the exactness 
proved in the previous paragraph, so they have the same value of $\a(\T)$ as 
well. Therefore it is enough to prove that $\a(\T)=-1$ for a particular such 
element.

For $n=1$, the claim can be verified using the example following 
Lemma~\ref{L4}. Hereafter we assume $n>1$, and we construct an 
element $\T\in\BD$ as in~\eqref{11}, by making a particular choice 
of $A,A'$.

Let $K_0\in \Delta/\Gamma $ be the coset  $\sm 100n \Gamma$. Clearly $UK_0\ne 
K_0$ and $SK_0\ne K_0$, so we can take $A,A'$ in~\eqref{10} to be given by
$$ A\= (1+U+U^2) \sm 100n +\cdots\,, \quad A'= (1+S) \sm 100n+\cdots\,,$$
where the part of $A,A'$ supported on other cosets than $K_0, SK_0=U^2 K_0, 
UK_0$ can be chosen such that~\eqref{10} is satisfied. We have
  $$ A_{K_0}\=A'_{K_0}\= \sm 100n\,,
 \quad A_{SK_0}\=U^2 A_{K_0}\=S A_{K_0} T^{-n}\,,
 \quad A'_{SK_0}\=S A_{K_0}\,.$$ 
Using that $1-T^{-n}= (1+T^{-1}  \+ \cdots \+ T^{-n+1})
  \,\bigl[(1-S)+S(1-U^2)\bigr]$ we get 
  $$A'_{SK_0} \m A_{SK_0} 
  \=S A_{K_0}(1+T^{-1}  \+ \cdots \+ T^{-n+1})
  \,\bigl[(1-S)+S(1-U^2)\bigr]\;,$$ 
which implies that 
$B_{SK_0}(1-S)= S A_{K_0}(1+T^{-1}+\cdots+T^{-n+1})(1-S)$
in~\eqref{11},  and we conclude:
$$
\T_{K_0}\m S \T_{SK_0} \= \sm 100n \bigl[1-T^{-n} \m 
(1+T^{-1}+\cdots+T^{-n+1})(1-S)\bigr] .
$$
On the other hand, by Theorem \ref{T1} we know that $\T-\a \wT_\DD\in\I$ for 
some $\wT_\DD$ satisfying~\eqref{Aa} and a number $\a=\a(\T)$. Taking only 
that part of~\eqref{Aa} supported on matrices in $K_0$ we have (using that 
$TK_0=K_0$)
$$
\T_{K_0}\m S \T_{SK_0} \m \a(\T) \sum_{b \!\!\!\!\pmod{n}} \sm1b0n(1-S)\ 
\inn (1-T)\R\,.
$$
From the last two equations it follows that 
$(1+\a(\T))\cdot\sum_{b \!\!\pmod{n}}\sm1b0n(1-S)\in (1-T)\R$.
Since the matrices in the last sum are in distinct $\G_\infty$-orbits 
in $\G_\infty\bsh\DD$, from~\eqref{T} we obtain that $\a(\T)=-1$.
\end{proof}

For the proof of Theorem~\ref{T0}, we only need the following immediate
consequences of Theorem~\ref{T2}.
\begin{cor} \label{C1} Let $\DD=\M_n$ and $\T\in \BD$, that is, 
$\T$ satisfies property~\eqref{B} in the introduction. 
We have $\la \T, K \ra =-1$ for all cosets $K\in\DD\bsh\G$ if and only if
$\T$ satisfies~\eqref{A}.    
\end{cor}
\begin{cor}\label{C2} Let $\DD=\M_n$ and $\T=\wT_n\in \R_n$ 
an element satisfying both~\eqref{A} and~\eqref{B}. 
Then for each conjugacy class $X\subset \DD$ the quantity
$\la \T, X \ra$ depends only on $X$ and not on the choice of~$\T$.
\end{cor}
\begin{proof}[Proof of Corollaries] We decompose $\M_n=\bigsqcup \DD'$ into a 
finite
disjoint union of double cosets $\DD'$. For $\T\in\BD$, 
we have a corresponding decomposition $\T=\sum \T_{\DD'}$.
Clearly $\T$ satisfies~\eqref{A} if and only if all such $\T_{\DD'}$ 
satisfy~\eqref{A}, and the same holds for the 
condition $\la \T,K \ra =-1$ for all cosets $K\in\DD\bsh\G$. Therefore
Corollary~\ref{C1} follows from part (a) of Theorem~\ref{T2}. 

If two elements $\T, \T'\in\BD$ satisfy~\eqref{A},  by the 
exact sequence following Theorem~\ref{T2} we 
have $\T_{\DD'}-\T'_{\DD'}\in\J_{\DD'}$ for each double coset 
$\DD'\subset\DD$, so $\T-\T'\in\JD$. But $\JD$ is spanned by 
$M-\g M \g^{-1}$  for $\g\in \G, M\in \DD $, so  
$\la \T,X \ra=\la \T', X\ra$ for each conjugacy class $X$, proving 
Corollary~\ref{C2}. 
\end{proof}

\section{An explicit Hecke operator}\label{S3}

Explicit elements $\wT_n\in \Q[\M_n]$ satisfying condition~\eqref{A} 
were first given by Manin~\cite{M73} using continued fractions 
(as re-interpreted in~\cite{Z2}, where this condition was introduced),
and other constructions were given in~\cite{Z} and~\cite{Me}.
In this section we prove Theorem~\ref{T0} by giving an explicit element 
satisfying all \emph{three} 
properties~\eqref{A}, \eqref{B} and~\eqref{C}. 
Since property~\eqref{C} is the hardest to prove, we start with an element
that satisfies it by construction, and then show that it verifies~\eqref{B}
and property~(i) of Theorem~\ref{T0} as well. The corollaries at 
the end of the 
previous section then show that it satisfies~\eqref{A} as well.  

Since we want to give a uniform formula for all $n$, it is convenient 
to introduce $\wR=\bigotimes_n\R_n$, the vector space of infinite 
formal linear combinations of elements of~$\R$ with only finitely many 
elements of any fixed determinant. We look for~$\wT$ of the form
$\wT\= -E\+H \+\sum c_i(M_i-\g_iM_i\g_i^{-1})$, 
where $E\in \wR$  contains representatives of all elliptic, and $H\in\wR$ 
of all hyperbolic split together with scalar conjugacy classes in~$\M$, 
and  such that $\la -E+H,X \ra=w(X)$ for all conjugacy classes 
$X\subset \M$, with $w(X)$ 
defined in the introduction.  Any such $\wT$ satisfies~\eqref{C}.

Choosing representatives for elliptic conjugacy classes amounts to choosing a 
fundamental domain of $\G$ acting on the upper half plane $\H$. 
Let~$\chi$ be the characteristic function of the fundamental domain 
$\F$ shown in Figure~\ref{F1},
\begin{figure}[h]
\begin{tikzpicture}[scale=1.5]
   \draw[->] (0,0) -- (0,2) node[anchor=east] {$y$};
   \draw[->] (0,0) -- (1.5,0) node[anchor=south] {$x$};
   \draw (0,0) arc (180:120:1);
   \draw (0.5,0.866)--(0.5,2);
   \node at (0.7,0.9) {$\rho$};
   \node[draw,circle,inner sep=1pt,fill] at (0.5,0.866){};
   \node[draw,circle,inner sep=1pt,fill] at (0,1){};
   \node at (-0.2,1) {$i$};
   \node at (0.3, 1.5) {$\F^+$};
   \node at (0.2, 0.8) {$\F^-$};
   \draw[dashed] (0,1) arc (90:60:1);
\end{tikzpicture}
\caption{ The fundamental domain $\F=\{z\in \H\;:\; 0\le \re z\le\dfrac 12, 
|z-1|\ge 1 \}$.}
\label{F1}
\end{figure}
modified on the boundary by setting $\chi(z)$ equal to $1/2\pi$ times 
the angle subtended by $\F$ at~$z$ (i.e.,~$\chi$ is 1 in the interior 
of $\F$, 0 outside of $\F$, 1/2 on the boundary points different from the 
corner~$\rho=e^{\pi i/3}$, and 1/3 at $\rho$).  Our choice of $\wT$ is then
\be\label{12}
  \wT\= -E\+H \+ X-S X S\+Y-U^2 Y U \+Z-U^2 Z U\,,\ee
with $E,H,X,Y,Z\in\wR$ defined by 
$$E\=\sum_{M\ \mathrm{elliptic}} \chi(z_M) M\,=\,\smm{0\le a-d\le 
-b;\; a-d\le c}{b<0<c}\,,  \qquad 
  H\= \smm{a-d\le -b\le c}{c\=0<a}\,,
  $$
 \[ X\=\smm {0\le a-d;\; -b \le c}{ b<0<d}\,,\quad 
Y\=\smm {a-d\le -b \le c}{0<c<a}\,, \quad 
Z\=\smm {a-d\le c\le -b}{0<a;\; 0<c}\,. 
\]
Here $z_M$ is the unique fixed point in $\H$ of an elliptic matrix~$M$, and the 
notation $\la \#\ra$, where $\#$ is a collection of inequalities written
on two lines, means the sum $\sum c(M) M \in \wR$ over matrices 
$M=\sm abcd$ with entries satisfying these inequalities, weighted with a
coefficient $c(M)\in\{1,\,\frac 12,\, \frac14,\, \frac13,\,\frac16 \}$ 
according to the number of inequalities in the \emph{first line} of~$\#$
that become equalities for~$M\,$: it is~1 if there are no 
equalities, $\frac12$ if there is exactly one, and $\frac14$, $\frac13$
or~$\frac16$ if there are two and they are independent ($A\le B$, 
$C\le D$), overlapping ($A\le B$, $A\le C$), or nested ($A\le B\le C$), 
respectively. Note that in these definitions the first line of~$\#$ involves
only the coefficients of the quadratic form $Q_M= [c, d-a,-b]$ 
associated to~$M$. Note also that the definition of the coefficient 
$c(\,\cdot\,)$ behaves correctly (additively) if $\#$ splits up as a union 
of two sets of inequalities, as we will use several times below. 
Recall that elements in $\M$ are defined up to $\pm 1$; we always choose 
representatives with non-negative lower left entry.

As already explained, the above choice of~$\wT$ automatically 
satisfies property~\eqref{C}. The difficult part was to find a choice 
of elements $X,Y,Z$ for which it satisfies~\eqref{B} as well, and this 
particular choice was found numerically with the help of a computer. 
Before checking property~\eqref{B}, we first show that 
formula~\eqref{12} hides considerable cancellation. 

\begin{lemma}\label{L5} \emph{(a)} The element $\wT$ belongs to $\wR$, namely 
$\wT=\sum_{n=1}^\infty \wT_n$, where $\wT_n\in\Q[\M_n]$.

   \emph{(b)} The element $\wT$ simplifies to
\be\label{13}
\wT\=\smm {a-d\le -b\le c}{ 0\le c<a}-\smm{-b\le a-d\le c}{b<d\le 0} 
-\smm{0\le a-d\le c \le -b}{a\le 0< c}
-\smm{0\le a-d\le -b\le c}{d\le 0<-b}.
\ee
\end{lemma}

\begin{proof}
(a) The terms $E$ and $H$ clearly contain finitely many elements of a given
determinant. For $X$ and $Y$, and for $Z$ in the case $d>a$, we have
that $a,d>0$, and $c, -b>0$, so $ad-bc=n$ has finitely many solutions for 
each $n$. For $Z$ in the case $0\le a-d$ the corresponding matrix is elliptic 
with fixed point in $\F$, so the same conclusion holds. 

(b) Splitting $\F$ into the parts $\F^+$, $\F^-$ for which $|z|\ge 1$, $|z|\le 1$, 
as in Figure~\ref{F1}, we have
\[
E\= \smm {0\le a-d\le c \le -b}{0<c}\+\smm {0\le a-d\le -b \le 
c}{b<0}\,=:\,A+B\,.
\]
We decompose: $A=A_1+A_2$, the parts of $A$ for which $a>0$, respectively $a\le 
0$; $B=B_1+B_2$, the parts of $B$ for which $d>0$, respectively $d\le 0$;
$Z=Z_1+Z_2$, the parts for which $a\le d$, respectively $a\ge d$ (with boundary 
coefficients as indicated below); and $U^2 Z U=\smm {-b\le a-d \le  c}{b<0;\; 
b<d}=Z_3+Z_4$, the parts for which $d>0$, respectively $d\le 0$. We check
easily that
\[Z_1\=SXS\=\smm {a-d\le 0;\; c\le -b }{0<a;\; 0<c}\,, 
\quad Z_2\=A_1\=\smm{0\le a-d\le c \le -b}{0<a;\; 0<c}\,,\]
\[X\=B_1+Z_3+U^2 Y U \=  \smm{0\le a-d\le -b \le c}{b<0<d}\+
\smm{-b\le a-d \le  c}{b<0<d}\+\smm {-b\le c\le a-d}{b<0<d}.\] 
 The term $H$ in~\eqref{12} can be absorbed in the 
sum~$Y$ as the boundary term for which $c=0$, yielding the first term 
in~\eqref{13}, 
and~\eqref{12} simplifies to~\eqref{13}, 
where the second, third, and fourth terms are $Z_4$, $A_2$, and $B_2$, 
respectively. 
\end{proof}

\begin{theorem}\label{T4}
\emph{(a)} For each $n\ge 1$ the element $\wT$ satisfies
  \[ \wT\pi_S\in \pi_U \wR\,, \quad \wT\pi_U\in \pi_S \wR\,, \]
namely each component $\wT_n$ satisfies~\eqref{B}.

\emph{(b)} We have $\la \wT,K \ra=-1$ for all right cosets $K=M\G\in\M/ \G$.
\end{theorem}
Theorem~\ref{T0} immediately follows: by construction, each 
component $\wT_n$ of $\wT$ satisfies~\eqref{C}, and the previous theorem and 
Corollary~\ref{C1} show that it satisfies~\eqref{A} as well. Corollary~\ref{C2} 
then shows that \emph{any} element satisfying~\eqref{A}, 
\eqref{B} also satisfies \eqref{C}.

\begin{proof} (a) Write $\wT=T_1-T_2-T_3-T_4$ in~\eqref{13}. We will show more precisely 
that 
\be\label{14}\begin{aligned}
\wT &\= -F-UFS-U^2F+T_1(1-S)\\
\wT &\=-G-SGU+ T_1(1-U)+\frac 1{12} \smp{a-d=c=-b}{d\le 0<a}(1-U)\,,
\end{aligned}
\ee
so that $\wT \pi_S=-(1+U+U^2)F\pi_S$, $\wT \pi_U=-(1+S)G\pi_U$, and part (a) follows.
Here 
$$
F\=\smm{0\le a-d\le c;\; a-d \le -b}{d \le  b< 0 <c}\,,\quad 
G\=\smm{0\le a-d\le c;\; a-d \le -b;\; a\le -b-c}{a \le 0;\; b<0<c }^{*}\,,
$$ 
with $\la\#\ra^*\in\wR$ denoting the sum of matrices  $\sm abcd$, weighted by $1/2$
if at least one inequality on the first line of~$\#$ becomes equality, except that matrices with 
 $a-d=c=-b$ or $a=d=-b-c$ are weighted by 1/4, and $[\#]\in\wR$ denoting the 
 corresponding unweighted sum of matrices.

For a sum $\xi \in \wR$ given by inequalities $\#$ on two lines, 
as in $\la\#\ra$ or $\la\#\ra^*$ above,  we denote by $\xi^\bd$ 
its ``interior'', that is, the subsum of the terms for which at most one inequality on the 
first line becomes equality (therefore the terms are weighted by either 1 or 1/2 if
there is no equality or exactly one equality, respectively).  

To prove the first formula, we decompose $T_3=T_{3,1}+T_{3,2}$, $T_4=T_{4,1}+T_{4,2}$
where $T_{3,1}$,  $T_{4,1}$ and $T_{3,2}$, $T_{4,2}$ have the 
added inequality $b<d$ and $d\le b$, respectively. We have $F\=T_{3,2}+T_{4,2}$ and
\[T_{3,1}\=\smm {0\le a-d\le c\le -b}{a\le0 <c;\; b<d}
,\quad T_{4,1}+T_2\= \smm{0\le a-d\le c;\; -b\le c}{b<d\le0}=:T_5\,,
\]
so $\wT=T_1-T_{3,1}-T_5-F$. We also have 
\[ T_1 S\=\smm {d\le a\le b+c}{b<d\le 0},\ \
UFS =\smm {d\le a \le c+d;\; b+c\le a  }{a\le0<c;\; b<d},\ \
U^2F=\smm {a-d\le c ;\; -b \le c \le a-b}{b<d\le 0<a }\,, \]
and we now decompose the terms by adding inequalities on the first line as follows: $T_1S=T_{1,1}+T_{1,2}$, where $T_{1,1}$, $T_{1,2}$ are obtained by imposing the extra inequality $-b\ge  c$,
respectively $-b\le c$; and $T_5=T_{5,1}+T_{5,2}$,  $T_{3,1}=T_{3,1a}+T_{3,1b}$,
with the range of summation split according as $a\le b+c$ (for $T_{5,1}$ and  $T_{3,1a}$),
or $a\ge b+c$ (for $T_{5,2}$ and $T_{3,1b}$). We check that $T_{1,1}^\bd=T_{3,1a}^\bd$, 
$T_{1,2}^\bd=T_{5,1}^\bd$, and 
\[ \begin{aligned}
T_5\=T_{5,1}^\bd+T_{5,2}^\bd\quad &+\frac13 \smp{a-d=c=-b}{b<d\le 0}
+\frac14 \smp{a=d;\;  c=-b}{b<d\le 0}
+\frac12 \smp{a=0;\;  c=-b}{b<d<0}\\
T_{3,1}\=T_{3,1a}^\bd+T_{3,1b}^\bd&
+\frac14 \smp{a=d ;\;  c=-b}{b<d\le 0}+\frac12 \smp {a=d=b+c}{b<d\le 0}\\
T_1S\=T_{1,1}^\bd+T_{1,2}^\bd\quad& 
 +\frac12 \smp{a=d;\;  c=-b}{b<d\le 0}+\frac16 \smp {a=d=b+c}{b<d\le 0}
+\frac12\smp{a=0;\;  c=-b}{b<d<0}  \\
UFS\=T_{3,1b}^\bd\quad\!\!\phantom{+T_{3,1b}^\bd } &
+\frac13  \smp {a=d=b+c}{b<d\le 0} \\
U^2F\=T_{5,2}^\bd\quad\phantom{+T_{3,1b}^\bd }& + \frac13 \smp{a-d=c=-b}{b<d\le 0}
\end{aligned}  
\]    
(in the sum $T_{3,1b}^\bd$ there is an ambiguity whether to include the term with $a=0=b+c$,
and we choose to include it, by taking the first line inequalities to be $0\le a-d\le c\le a-b$; the inequality $c\le -b$, 
which was part of the first line of $T_{3,1}$,  is then implied by $a\le 0$ on the second line). 
We obtain that $ T_{3,1}+T_5 = T_1S+UFS+U^2F$, so the first equation in~\eqref{14} follows.

To prove the second equation, we first note that $T_2=T_1 U$, 
so $\wT= -T_3-T_4 +T_1(1-U)$. 
We decompose $T_4=T_{4,1}+T_{4,2}$, the subsums for which we add $a\le -b-c$, 
respectively $ -b-c\le a$ on the first line of~$T_4$. We have 
$SGU= \smm{0\le a-d\le -b\le c;\;  -b-c\le a}{-c<d\le 0<-b}^{*}$, 
and we check the following relations
\[ \begin{aligned}
T_3\=T_3^\bd\phantom{\ \ \ +T_{4,2}^\bd} &+\frac16\smp{a-d=c=-b}{a\le0<c}+\frac14\smp{a=d;\;  c=-b}{a\le 0<c}\\
T_4\=T_{4,1}^\bd+T_{4,2}^\bd&+\frac16\smp{a-d=c=-b}{d\le0<c}+
\frac14\smp{a=d;\;  c=-b}{d\le 0<c}+\frac12 \smp{a=d=-b-c}{d<0<-b} +
\frac12 \smp{a=0;\;  c=-b}{0<-d<c}\\
G\=T_3^\bd+ T_{4,1}^\bd\  \ &+\frac14 \smp{a-d=c=-b}{a\le0<c}+
\frac12\smp{a=d;\; c=-b}{a< 0<c}+ \ \frac14 \smp{a=d=-b-c}{d\le 0<-b}\\
SGU\= T_{4,2}^\bd\phantom{\ \ +T_{4,2}^\bd}&+\frac14\smp{a-d=c=-b}{d\le0<a}
 + \frac14 \smp{a=d=-b-c}{d\le 0<-b}+ \frac12 \smp{a=0;\; c=-b}{ 0<-d<c},
\end{aligned}  
\]
hence $T_3+T_4-G-SGU$ involves only the  four terms with $a-d=c=-b$  above.   Moreover 
$$\smp{a-d=c=-b}{d\le 0<c}=\smp{a-d=c=-b}{a\le 0<c}+\smp{a-d=c=-b}{d\le 0<a}, 
\quad \smp{a-d=c=-b}{d\le 0<a}\cdot U= \smp{a-d=c=-b}{a\le 0<c}\,,
$$
and we conclude that 
$ T_3+T_4-G-SGU=\frac 1{12} \smp{a-d=c=-b}{d\le 0<a}(U-1)$, yielding
the second line in~\eqref{14}.

(b) We have a decomposition of $\M$ into a disjoint union of
double cosets
  \[ \M\=\bigcup_{a,m\ge 1} \G \sm a00{am} \G\,, \]
and a corresponding decomposition $\wT=\sum_{a,m}\wT_{a,m}$, with
$\wT_{a,m}$ the part of $\wT$ supported on the corresponding double coset.
Each $\wT_{a,m}$ satisfies~\eqref{B} by part (a), and $\wT_{a,m}=\sm a00a 
\wT_{1,m}$,
so by Theorem~\ref{T2} it is enough to show that $\la 
\wT_{1,m},K_0  \ra=-1$ for all $m$, for the particular coset $K_0=\sm 100m  \G$. 

For $m=1$ we have $\wT_{1,1}=\wT_1=1-\pi_S-\pi_U$, and the claim is clear. 
Assuming therefore $m>1$, we have to find the matrices $M=\sm ab{mc}{md} \in 
K_0$ with $a,b,c,d\in\Z$, $ad-bc=1$, in each of the four terms in~\eqref{13}. 
For the first two terms we have  $\la T_1-T_2, K_0 \ra=0$, as $T_2=T_1U$. 
For the sum $T_3$ we write $m=\det M$ as follows 
\[m=mc(-b)+(-a)(a-md)+a^2\ge m^2c^2\ge m^2\,,
\]
so there are no matrices from $K_0$ in the sum $T_3$. For $T_4$ we write 
\[m=mc(-b)+md(a-md)+m^2d^2\ge b^2+m^2d^2-m|db|
\ge 3m^2d^2/4\,,
\]
so $d=0$, $c=-b=1$, and $0\le a \le 1$. Therefore $\la T_4, K_0\ra =1/2+1/2=1$, so
$\la \wT_{1,m}, K_0 \ra = -1$.
\end{proof}

We already pointed out in the introduction that Theorem~\ref{T0}
immediately implies the classical  Kronecker-Hurwitz class number formula. In fact 
the proof of Theorem~\ref{T4} implies a refinement of this class number formula, 
similar in spirit to the refinement proved by different means in~\cite{PZ}. More precisely,
in~\cite{PZ} the class number formula was proved
by establishing a correspondence between an easily 
countable subset of the right cosets $\M_n/\G$ and half the elliptic conjugacy 
classes of matrices in~$\M_n$, whereas now we obtain a ``weighted'' bijection
between \emph{all} right cosets and an easily countable subset of 
the elliptic conjugacy classes, namely those conjugacy classes containing the
matrices in the third and fourth sums in~\eqref{13}. We indicate the argument briefly.

Define a weighting function $\a $ on $\M$ by setting $\a(M)=0$ for 
non-elliptic~$M$, while $\a(M)$ for an elliptic matrix~$M=\sm abcd \in \M$ 
with $c>0$ is given by 
\[ \a(M)=\chi^+(z_M)\cdot \dd(a)+\chi^-(z_M)\cdot \dd(d)\,, \]
where $\chi^\pm$ are the characteristic functions of the half-fundamental domains 
$\F^\pm$ in Figure~\ref{F1} (having boundary values given as for $\chi$ 
in terms of the angle subtended), and $\delta$ is the characteristic function of 
the set of non-positive integers. Conjugating by  $\sm 0110\in\GL_2(\Z)$  and 
simultaneously changing the sign of the matrix 
to preserve the condition $c>0$
interchanges matrices with fixed point in $\F^+$ with those with fixed point 
in $\F^-$ and interchanges $a$ and $-d$ in the formula for $\a$, so we have
\be \label{18}
\begin{aligned}
\sum_{M\in \M_n } \a(M) &\=  \sum_{M\in \M_n}\chi^-(z_M)\,\bigl(\dd(-d)+\dd(d)\bigr) \\
&\= \sum_{M\in \M_n}\chi^-(z_M) \; +
              \sum_{\substack{  0\le a\le -b\le c \\n=-bc}} \a \bigl(\sm abc0\bigr) \\
&\=\frac12  \sum_{t^2\le 4n}H(4n-t^2)\,+\, \frac12 \sum_{\substack{n=bc\\ b>0}} \min(b,c)\;,
\end{aligned}
\ee
where the final equality follows from the well-known $\G$-equivariant 
bijection between matrices and binary quadratic forms. (The calculation of the 
final sum must be modified slightly for  the terms with $-b=c$ when $n$ is a square, 
but this is compensated by the term $H(0)=-1/12$.) 
The Kronecker-Hurwitz formula can be rephrased by saying that the expression in the last line of~\eqref{18} 
equals $\sigma_1(n)=|\M_n/\G|$.   The above-mentioned refinement is then:  
\begin{prop}\label{P5}
The sum  $\sum_{M\in K}\a(M)$ equals 1 for each right coset $K\in \M/\G$.
\end{prop} 
\begin{proof} In the proof of Theorem~\ref{T4} we saw that 
$\wT=T_1(1-U)-T_3-T_4$, so part (b) of the theorem gives 
$\la T_3+T_4,K \ra =1$ for each right coset $K\in \M\bsh \G$. 
But $T_3+T_4=\sum_{M\in\M}\a(M) M$ by the very definition of $\a$. 
\end{proof}

\noindent {\bf Acknowledgements.} The first author was supported in part by 
CNCS-UEFISCDI grant PN-III-P4-ID-PCE-2016-0157. He would like to thank the MPIM in Bonn
for providing support and a stimulating research environment during several visits
while working on this paper.


\begin{thebibliography}{MM}
\bibitem{Ba} H. Bass, \emph{Euler characteristics and characters of 
discrete groups.} Inv. Math. 35 (1976), 155--196

\bibitem{Br} K.S. Brown, \emph{Complete Euler Characteristics and fixed-point 
theory.} J. Pure and Applied Algebra 24 (1982), 103--121

\bibitem{CZ} Y.J. Choie and D. Zagier, \emph{Rational period functions for
PSL(2,Z).} Contemporary Math. 143 (1993), 89--108

\bibitem{GM} M. Goresky, R. MacPherson, \emph{The topological trace formula}. 
J. reine angew. Math. 560 (2003), 77--150

\bibitem{M73} Ju. I. Manin,  \emph{Periods of parabolic forms and p-adic 
Hecke series.} Mat. Sb. 21 (1973), 371--393

\bibitem{Me} L. Merel, \emph{Universal Fourier expansions of modular forms}. On Artin's
conjecture for odd 2-dimensional representations, Springer (1994), 59--94 

\bibitem{PP} V. Pa\c{s}ol, A.A. Popa, \emph{Modular forms and period polynomials.}
Proc. Lond. Math. Soc. 107/4 (2013), 713--743

\bibitem{P} A.A. Popa, \emph{On the trace formula for Hecke operators on congruence subgroups}.
Proc. Amer. Math. Soc. 146/7 (2018), 2749--2764

\bibitem{P2} A.A. Popa, 
\emph{On the trace formula for Hecke operators on congruence subgroups, II}.
Research in the Math. Sciences (2018) 5:3

\bibitem{PZ} A.A. Popa, D. Zagier, \emph{A combinatorial refinement of the 
Kronecker-Hurwitz class number relation}. Proc. Amer. Math. Soc. 145/3 (2017), 
1003--1008

\bibitem{Z2} D. Zagier, \emph{Hecke operators and periods of modular forms.}
Israel Math. Conf. Proc. 3 (1990), 321--336

\bibitem{Z1}D. Zagier, \emph{Periods of modular forms and Jacobi theta 
functions.} Invent. math. 104 (1991), 449--465

\bibitem{Z} D. Zagier, \emph{Periods of modular forms, traces of
Hecke operators, and multiple zeta values}, in Research on automorphic forms and
$L$-functions, RIMS K\^oky\^uroku 843 (1993), 162--170 
\end{thebibliography}
\end{document}